\documentclass[12pt,leqno,a4paper]{amsart}

\usepackage{amssymb,amsthm,enumerate}
\usepackage{hyperref}                         
\usepackage[all]{xy}

\usepackage{tikz}
\usetikzlibrary{angles}
\usetikzlibrary{decorations.markings,calc}
\usetikzlibrary{cd}

\hypersetup{pdftex,                           
bookmarks=true,
pdffitwindow=true,
colorlinks=true,
citecolor=black,
filecolor=black,
linkcolor=black,
urlcolor=blue,
hypertexnames=true}

\newcommand{\IZ}{{\mathbb{Z}}}

\newcommand{\fp}{{\mathfrak{p}}}     

\newcommand{\cO}{{\mathcal{O}}}

\DeclareMathOperator{\End}{End}               
\DeclareMathOperator{\Hom}{Hom}             
\DeclareMathOperator{\Ker}{Ker}
\DeclareMathOperator{\GL}{GL}
\DeclareMathOperator{\chap}{Cap}

\newcommand{\lconj}[2]{\,^{#1}\!#2}                     

\let\lra=\longrightarrow

\let\wh=\widehat

\newtheorem{thm}{Theorem}[section]
\newtheorem{lem}[thm]{Lemma}

\newtheorem{cor}[thm]{Corollary}

\theoremstyle{definition}

\theoremstyle{remark}
\newtheorem{rem}[thm]{Remark}

\begin{document}


\title[On the lifting of the Dade Group]
      {On the Lifting of the Dade Group}

\date{\today}

\author{Caroline Lassueur and Jacques Th\'{e}venaz}
\address{Caroline Lassueur\\ FB Mathematik, TU Kaiserslautern, Postfach 3049, 67653 Kaisers\-lautern, Germany.}
\email{lassueur@mathematik.uni-kl.de}
\address{Jacques Th\'{e}venaz\\ EPFL, Section de Math\'{e}matiques, Station 8, CH-1015 Lausanne, Switzerland.}
\email{jacques.thevenaz@epfl.ch}



\begin{abstract}
For the group of endo-permutation modules of a finite $p$-group, there is a surjective reduction  homomorphism from a complete discrete valuation ring of characteristic~0 to its residue field of characteristic~$p$.
We prove that this reduction map always has a section which is a group homomorphism.
\end{abstract}

\maketitle


\pagestyle{myheadings}
\markboth{C. Lassueur and J. Th\'{e}venaz}{On the lifting of the Dade group}


\section{Introduction}

\noindent The Dade group and endo-permutation modules are important invariants of block theory of finite groups. For instance, they occur in the description of source algebras of blocks (see e.g. \cite[\S50]{ThevenazBook} or \cite[\S50]{CEKL}), or as sources of simple modules for $p$-soluble groups (see e.g. \cite[\S30]{ThevenazBook}). They also play an important role in the description of equivalences between block algebras, such as derived equivalences in the sense of Rickard or Morita equivalences (see e.g. the recent papers \cite{KL17,BKL18}).  The final classification of endo-permutation modules was obtained by Bouc in~\cite{Bouc}, but one last question about the structure of the Dade group remained open, namely whether lifting endo-permutation modules from positive characteristic to characteristic zero can be turned into a group homomorphism. The aim of this note is to fill this gap.
\par
Throughout $p$ denotes a prime number, $P$ a finite $p$-group, and $\cO$ a complete discrete valuation ring of characteristic~$0$ containing a root of unity of
order $\exp(P)$, the exponent of~$P$, with a residue field $k:=\cO/\fp$ of characteristic~$p$, where $\fp=J(\cO)$ is the unique maximal ideal of $\cO$.
We let $R\in\{\cO,k\}$.
All modules considered are assumed to be finitely generated left modules, and we  will consider $\cO P$-lattices only, that is $\cO P$-modules which are free as $\cO$-modules.   For an $\cO P$-lattice $L$, the reduction modulo $\fp$ of $L$ is the $kP$-module $L/\mathfrak{p}L$\,, and a $kP$-module $M$ is said to be \emph{liftable} if there exists an $\cO P$-lattice $\wh{M}$ such that $M\cong \wh{M}/\fp\wh{M}$. 
\par
Very few classes of modules are known to be liftable from $k$ to $\cO$ in general. However, it is known that any endo-permutation $kP$-module can be lifted to an endo-permutation $\cO P$-lattice.
This nontrivial result is a consequence of their classification, due to Bouc~\cite{Bouc}. 
Let us fix some more precise notation.
Let $D_R(P)$ denote the group of endo-permutation $RP$-lattices (i.e. the Dade group of $P$).
The reduction homomorphism modulo~$\fp$
$$\pi_p : D_\cO(P) \longrightarrow D_k(P)$$
maps the equivalence class of an endo-permutation $\cO P$-lattice to the equivalence class of its reduction modulo~$\fp$.
By a main result of Bouc \cite[Corollary~8.5]{Bouc}, the map $\pi_p$ is surjective.
Moreover, its kernel is isomorphic to the group $X(P)$ of one-dimensional $\cO P$-lattices (see Lemma~\ref{lem:one-dim}).
The aim of this note is to prove that this reduction map always admits a section which is a group homomorphism.

\begin{thm} \label{main} Let $P$ be a finite $p$-group.
\begin{enumerate}[ \rm(a)]
\item The group homomorphism $\pi_p : D_\cO(P) \to D_k(P)$ has a group-theoretic section.
\item $D_\cO(P) \cong X(P) \times D_k(P)$.
\end{enumerate}
\end{thm}
By the above, it is clear that (b) follows from~(a), so we only have to prove~(a). In other words, we have to show how to choose the lifts of all capped endo-permutation $kP$-module in a suitable fashion. When $p$ is odd, the result is easy and does not require any other deep result about endo-permutation modules. We will briefly recall this construction in Lemma~\ref{odd}. Thus the main question is to deal with the case of 2-groups in characteristic~2.  Furthermore, we will explain in Remark~\ref{rem:DadePalg} that, as a consequence of the surjectivity of~$\pi_p$, the result is equivalent to another result in terms of Dade $P$-algebras  mentioned without proof in \cite[Remark~29.6]{ThevenazBook}.  
\par
In fact, our aim is not only to prove that $\pi_p$ always admits a group-theoretic section, but also more accurately to describe how to define a section in a natural way on a set of generators of~$D_k(P)$.

\bigskip\bigskip
\section{Endo-permutation lattices and the Dade group}
\noindent
We start by recalling some basic facts about endo-permutation modules and the Dade group. We refer to the survey~\cite{ThevenazTour} for more details and suitable references.\par

An $RP$-module $M$ is called \emph{endo-permutation}  if its endomorphism algebra $\End_R(M)$ is a permutation $RP$-module,
where $\End_R(M)$ is endowed with its natural $RP$-module structure via the action of~$P$ by conjugation: 
$$\lconj{g}{\phi}(m)=g\cdot\phi(g^{-1}\cdot m)\quad \forall\,g\in P, \; \forall\, \phi\in \End_R(L)\text{ and }\forall\,m\in M \,.$$
Notice that, if $R=\cO$, then it is easy to see that any endo-permutation $\cO P$-module is necessarily free as an $\cO$-module, i.e. an $\cO P$-lattice, because $\cO$ is a PID.
Hence, in the sequel, we consider $RP$-lattices only.
In particular, the \emph{dimension} $\dim_R M$ of an $RP$-lattice~$M$ is the rank of $M$ viewed as a free $R$-module.
Moreover, writing  $M^*=\Hom_R(M,R)$ for the dual of the $RP$-lattice $M$, we have $\End_R(M)\cong M\otimes_R M^*$ as $RP$-lattices.  

\par
An endo-permutation $RP$-lattice $M$ is said to be \emph{capped} if it has at least one indecomposable direct summand with vertex~$P$, and in this case there is in fact a unique isomorphism class of indecomposable direct summands of $M$ with vertex~$P$, called the \emph{cap} of~$M$ and denoted by $\chap(M)$. Unfortunately, $\chap(M)$ may appear with a multiplicity as a direct summand of~$M$ and we shall need to avoid this.
Following~\cite[Definition~5.3]{Las}, we say that an endo-permutation $RP$-lattice $M$ is \emph{strongly capped} if it is capped and if $\chap(M)$ has multiplicity one as a direct summand of~$M$.
We note that the class of strongly capped endo-permutation $RP$-lattices is closed under taking duals and tensor products (see \cite[Lemma~5.4]{Las}).
\par
A strongly capped endo-permutation $RP$-lattice $M$ has dimension $\dim_R(M)$ prime to~$p$, because $\dim_R(\chap(M))$ is prime to~$p$
(see \cite[Corollary~28.11]{ThevenazBook}) and any other direct summand has dimension divisible by~$p$ since its vertex is strictly smaller than~$P$.
This is the reason why we shall use strongly capped endo-permutation $RP$-lattices rather than capped endo-permutation $RP$-lattices.
\par
Two capped endo-permutation $RP$-lattices $M$ and $N$ are \emph{equivalent} if there exist two capped permutation $RP$-lattices $S$ and $T$
such that $M\otimes_R S\cong N\otimes_R T$.
This happens if and only if $\chap(M)\cong\chap(N)$ and this defines an equivalence relation.
The Dade group is usually defined as the set of equivalence classes $[M]$ of capped endo-permutation $RP$-lattices~$M$.
Any class contains, up to isomorphism, a unique indecomposable endo-permutation $RP$-lattice, namely the cap of any element of the class.
This indecomposable endo-permutation $RP$-lattice is of course strongly capped, so that we can define the Dade group by restricting each class and using only strongly capped endo-permutation $RP$-lattices.
We do so and define the Dade group $D_R(P)$ as the set of equivalence classes $[M]$ of strongly capped endo-permutation $RP$-lattices~$M$,
endowed with the product $[M]\cdot[N] = [M\otimes_R N]$ induced by the tensor product over $R$. (This definition coincides with \cite[Corollary-Definition 5.5]{Las}.)
\par
We note that two strongly capped endo-permutation $RP$-lattices $M$ and $N$ are equivalent if there exist two strongly capped permutation $RP$-lattices $S$ and $T$ such that $M\otimes_R S\cong N\otimes_R T$.
The identity element is the class~$[R]$ of the trivial $RP$-lattice~$R$ and $[R]$ consists of all strongly capped permutation $RP$-lattices  (i.e. permutation $RP$-lattices $RX$, where $X$ is a basis of~$RX$ permuted by the action of~$P$, having a trivial direct summand~$R$ with multiplicity one, corresponding to a unique fixed point in~$X$).
If $L$ is an indecomposable endo-permutation $RP$-lattice, it is the cap of any element of its class and an arbitrary element of the class $[L]$ has the form $\chap(L)\otimes_R RX$ where $RX$ is a strongly capped permutation $RP$-lattice. 
\par
An $RP$-lattice $M$ is called \emph{endo-trivial} if $\End_R(M)\cong R\oplus Q$, where $Q$ is a projective $RP$-lattice (or equivalently a free $RP$-lattice because $P$ is a $p$-group).
Clearly any endo-trivial $RP$-lattice $M$ is endo-permutation and moreover
$$T_R(P):=\{[M]\in D_R(P)\mid \chap(M)\text{ is endo-trivial}\}$$
 is a subgroup of $D_R(P)$. For simplicity, $T_R(P)$ is called here the group of endo-trivial $RP$-lattices (but it is only isomorphic to the usual group of endo-trivial $RP$-lattices).
 \par 
We recall that the Dade group is known to be a finitely generated abelian group, hence a product of cyclic groups. The most important examples of indecomposable endo-permutation $RP$-lattices are given by the relative Heller translates of the trivial $RP$-lattice, which we denote by  $\Omega^m_{P/Q}(R)$ ($m\in\IZ$, $Q\leq P$) as usual, and simply write $\Omega^m(R)$ when $Q=1$.
In fact, when $p$ is odd, $D_k(P)$ is generated by such modules, by \cite[Theorem~7.7]{Bouc}.\par

The group homomorphism $\pi_p:D_\cO(P) \to D_k(P)$ was defined in the introduction and we recall that it is surjective by \cite[Corollary~8.5]{Bouc}.

\begin{lem}\label{lem:one-dim}
The kernel of $\pi_p:D_\cO(P) \to D_k(P)$ is isomorphic to the group $X(P)$ of one-dimensional $\cO P$-lattices.
\end{lem}

\begin{proof}
If $L$ is an indecomposable endo-permutation $\cO P$-lattice with $[L]\in \Ker(\pi_p)$, then $\dim_k(L/\fp L)=1$, hence $\dim_\cO(L)=1$.
If, conversely, $\dim_\cO(L)=1$, then the one-dimensional $kP$-module $L/\fp L$
must be trivial since there are no nontrivial $p^n$-th roots of unity in the field~$k$ of characteristic~$p$. Therefore $[L]\in \Ker(\pi_p)$.
\end{proof}

\begin{rem} \label{rem:DadePalg}
The endomorphism algebra $\End_R(M)$ of an endo-permutation $RP$-lattice~$M$ is naturally endowed with the structure of a so-called \emph{Dade $P$-algebra}, that is, an $\cO$-simple permutation $P$-algebra whose Brauer quotient with respect to $P$ is nonzero. Furthermore, there exists also a version of the Dade group, denoted by $D^{alg}_R(P)$, obtained by defining an equivalence relation on the class of all Dade $P$-algebras rather than capped endo-permutation $RP$-lattices, where multiplication is given by the tensor product over $R$. We refer to \cite[\S28-29]{ThevenazBook} for this construction. This induces  a canonical homomorphism
$$d_R:D_R(P)\lra D^{alg}_R(P) \,, \qquad [M]\mapsto [\End_R(M)]\,,$$
which is surjective by \cite[Proposition 28.12]{ThevenazBook}. The identity element of $D_R^{alg}(P)$ being the class of the trivial $P$-algebra $R$, it follows that the kernel of $d_R$ is isomorphic to $X(P)$ when $R=\cO$, whereas it is trivial when $R=k$. 
Now, reduction modulo $\mathfrak{p}$ also induces a group homomorphism 
$$\pi^{alg}_p:D^{alg}_\cO(P)\lra D^{alg}_k(P) \,, \qquad [A]\mapsto [A/\mathfrak{p}A]\,.$$ 
Because $\End_\cO(M)/\fp\!\End_\cO(M)\cong \End_k(M/\fp M)$ for any $\cO P$-lattice,  it follows that we have a commutative diagram with exact rows and columns:
\begin{center}
\begin{tikzcd}[column sep=huge]
 &  X(P)  \arrow[hook]{d} \arrow{r}  & 1\arrow[hook]{d} \\
X(P)   \arrow[hook]{r}  \arrow{d} & D_\cO(P) \arrow[two heads, "d_\cO"]{r}  \arrow{d}[swap]{\pi_p}  & D^{alg}_\cO(P) \arrow["\pi^{alg}_p"]{d} \\
1\arrow[hook]{r}    & D_k(P)  \arrow{r}{\cong}[swap]{d_k} &  D^{alg}_k(P)\\
\end{tikzcd} 
\end{center}
The injectivity of $\pi_p^{alg}$ follows from the commutativity of the bottom-right square because
$$\ker(\pi_p^{alg}d_\cO)=\ker(d_k\pi_p)=\ker(\pi_p)=X(P)$$
and its image under~$d_\cO$ yields $\ker(\pi_p^{alg})=d_\cO(X(P))=\{1\}$.
The surjectivity of $\pi_p$ implies that $\pi_p^{alg}$ is also surjective, hence an isomorphism, so that $D_k(P)\cong D_k^{alg}(P)\cong D_\cO^{alg}(P)$.
Therefore, finding a group-theoretic section of~$\pi_p$ is equivalent to finding a group-theoretic section of~$d_\cO$.
\end{rem}


\section{Determinant}

\noindent
Given an $\cO P$-lattice $L$, we may consider the composition of the underlying representation of $P$ with the determinant homomorphism $\det:\GL(L)\lra \cO^{\times}$. This is a linear character of $P$ and is called the \emph{determinant} of $L$. Given $g\in P$,  we write $\det(g,L)$ for the determinant of the action of~$g$ on~$L$. If $\det(g,L)=1$ for every $g\in P$, that is, if the determinant of $L$ is the trivial character, then we say that $L$ is an $\cO P$-lattice of determinant $1$.\par

\begin{lem}\label{lem:det}
Let $L$ and $N$ be $\cO P$-lattices of determinant $1$.
\begin{enumerate}
\item[\rm(a)] $L^{*}$ is an $\cO P$-lattice of determinant $1$.
\item[\rm(b)] $L\otimes_{\cO}N$ is an $\cO P$-lattice of determinant $1$.
\end{enumerate}
\end{lem}

\begin{proof}
(a) Since the action of $g\in P$ on~$\varphi\in L^*$ is given by $(g{\cdot} \varphi)(x)=\varphi(g^{-1}x)$ for all $x\in L$, we have clearly
$$\det(g,L^*)=\det(g^{-1},L)=\det(g,L)^{-1} \,.$$
Since $L$ has determinant~1, so has $L^*$.\par

(b) The determinant of a tensor product satisfies the well-known property
$$\det(g,L\otimes_{\cO}N)=\det(g,L)^{\dim_{\cO} N}\cdot \det(g,N)^{\dim_{\cO} L} \,.$$
Since both determinants are~1, we obtain $\det(g,L\otimes_{\cO}N)=1$.
\end{proof}

\noindent 
Among the lifts of a strongly capped endo-permutation $kP$-module $M$, there always exists one which has determinant~1, by~\cite[Lemma~28.1]{ThevenazBook}, using our assumption that there are enough roots of unity in~$\cO$ and the fact that $\dim_k(M)$ is prime to~$p$ because $M$ is strongly capped.
This lift of determinant~1 is unique up to isomorphism and will be written~$\Phi_M$.

\begin{lem} \label{Phi}
Let $M$ and $N$ be strongly capped endo-permutation $kP$-modules. Then:
\begin{enumerate}[ \rm(a)]
\item $\Phi_M^* \cong \Phi_{M^*}$.
\item $\Phi_{M\otimes_k N}\cong \Phi_M\otimes_\cO \Phi_N$.
\item $\Phi_{M\otimes_k M^*}$ is a permutation $\cO P$-lattice lifting the permutation $kP$-module $M\otimes_k M^*$.
\end{enumerate}
\end{lem}

\begin{proof}
\begin{enumerate}[ \rm(a)]
\item It is clear that $\Phi_M^*$ lifts~$M^*$.
Since $\Phi_M$ has determinant~1, so has $\Phi_M^*$ by Lemma~\ref{lem:det}, and therefore $\Phi_M^* \cong \Phi_{M^*}$.
\item $\Phi_M\otimes_\cO \Phi_N$ has determinant 1 by Lemma~\ref{lem:det} and is therefore isomorphic to~$\Phi_{M\otimes_k N}$.
\item $\Phi_{M\otimes_k M^*} \cong \Phi_M\otimes_\cO \Phi_{M^*}$ by~(b).
Using (a), it follows that
$$[\Phi_{M\otimes_k M^*}] = [\Phi_M\otimes_\cO \Phi_{M^*}] =  [\Phi_M] \cdot [\Phi_{M^*}] =  [\Phi_M] \cdot [\Phi_M^*] = [\cO] \,,$$
which is the class consisting of all strongly capped permutation $\cO P$-lattices.
Therefore $\Phi_{M\otimes_k M^*}$ is a permutation $\cO P$-lattice.
\end{enumerate}
\end{proof}

These properties of the determinant allow us to prove Theorem~\ref{main}(a) in the odd characteristic case.
It is briefly mentioned without proof at the end of ~\cite[Remark~29.6]{ThevenazBook} that the map $d_\cO:D_\cO(P)\lra D^{alg}_\cO(P)$ always has a group-theoretic section when $p>2$,  which is equivalent to Theorem~\ref{main}(a) thanks to Remark~\ref{rem:DadePalg}.
For completeness, we provide a proof of this result in terms of lifts of modules.

\begin{lem} \label{odd}
Suppose that $p$ is an odd prime.
\begin{enumerate}[ \rm(a)]
\item Any permutation $\cO P$-lattice has determinant~1.
\item Let $[L]\in D_\cO(P)$. If $\chap(L)$ has determinant~1, then any element of the class $[L]$ also has determinant~1.
\item Let $M_1$ and $M_2$ be two indecomposable endo-permutation $kP$-modules and let $N$ be the cap of $M_1\otimes_k M_2$.
Then the cap of $\Phi_{M_1}\otimes_\cO \Phi_{M_2}$ is isomorphic to $\Phi_N$.
\item The map
$$D_k(P) \longrightarrow D_\cO(P) \,,\qquad [M] \to [\Phi_M]$$
is a well-defined group homomorphism which is a section of~$\pi_p$.
\end{enumerate}
\end{lem}

\begin{proof}
\begin{enumerate}[ \rm(a)]
\item Let $L=\cO X$ be a permutation $\cO P$-lattice, where $X$ is a basis of~$L$ permuted under the action of~$P$.
For any $g\in P$, the permutation action of~$g$ on~$X$ decomposes as a product of cycles of odd length, because the order of $g$ is odd. 
Any such cycle is an even permutation, so the determinant of the action of~$g$ on~$L$ is~1.\par
\item By the definition of the Dade group, an arbitrary element of the class $[L]$ has the form $\chap(L)\otimes_{\cO} \cO X$ where $\cO X$ is a strongly capped permutation $\cO P$-lattice.
Since both $L$ and $\cO X$ have determinant~1, so has their tensor product by Lemma~\ref{lem:det}(b).\par
\item Again by Lemma~\ref{lem:det}(b), the determinant~1 is preserved by tensor product. Hence the claim follows from (b).
\item This follows from (b) and (c).
\end{enumerate}
\end{proof}

From now on, unless otherwise stated, we assume that $P$ is a finite 2-group and $k$ has characteristic~2. 
It turns out that Lemma~\ref{odd} fails when $p=2$ in general.
It is clear that a (strongly capped) permutation $kP$-module always lifts in a unique way to a (strongly capped) permutation $\cO P$-lattice.
However, we emphasise that this lift may be different from the lift of determinant~1. 
It follows that two strongly capped endo-permutation $\cO P$-lattices in the same class in $D_\cO(P)$ need not have the same determinant. The problem is made clear through the following two results.

\begin{lem} \label{odd-permutation}
Let $g\in P\setminus\{1\}$ and let $C=\;<g>$.
\begin{enumerate}[ \rm(a)]
\item If $C=P$, then $\det(g,\cO P)=-1$.
\item If $C<P$, then $\det(g,\cO P)=1$.
\end{enumerate}
\end{lem}

\begin{proof}
\begin{enumerate}[ \rm(a)]
\item Since $P=\;<g>$, the action by permutation of~$g$ on~$P$ is given by a cycle of even length, hence an odd permutation.
Therefore $\det(g,\cO P)=-1$.
\item  Viewed by restriction as an $\cO C$-lattice, $\cO P$ is isomorphic to a direct sum of $|P:C|$ copies of~$\cO C$.
Since $C<P$, the index $|P:C|$ is even.
Therefore, using~(a), we obtain
$$\det(g,\cO P)=\det(g,(\cO C)^{|P:C|})=\det(g,\cO C)^{|P:C|}=(-1)^{|P:C|}=1 \,,$$
as was to be proved.
\end{enumerate}
\end{proof}

\begin{cor} \label{omega}
Let $g\in P\setminus\{1\}$ and let $C=\;<g>$.
\begin{enumerate}[ \rm(a)]
\item If $C=P$, then $\det(g,\Omega^1(\cO))=-1$.
\item If $C<P$, then $\det(g,\Omega^m(\cO))=1$, for any $m\in\IZ$.
\end{enumerate}
\end{cor}

\begin{proof}
\begin{enumerate}[ \rm(a)]
\item From the short exact sequence $0\to \Omega^1(\cO) \to \cO P \to \cO \to 0$, we obtain
$$\det(g,\Omega^1(\cO)) \det(g,\cO) = \det(g,\cO P) =-1 \,,$$
by Lemma~\ref{odd-permutation}. Since $\det(g,\cO) =1$, the result follows.
\item Recall that $\Omega^{-m}(\cO)\cong \Omega^m(\cO)^*$.
By Lemma~\ref{lem:det}, passing to the dual preserves the property that the determinant is~1. 
Therefore we may assume that $m> 0$ and we proceed by induction on $m$. 
Since $P$ is a 2-group, every projective $\cO P$-lattice is free, so there is a short  exact sequence of the form 
$$0\longrightarrow \Omega^m(\cO) \longrightarrow (\cO P)^r \longrightarrow \Omega^{m-1}(\cO) \longrightarrow 0$$
for some integer~$r$.
It follows that
$$\det(g,\Omega^m(\cO))\det(g,\Omega^{m-1}(\cO)) = \det(g,(\cO P)^r)=\det(g,\cO P)^r=1 \,,$$
by Lemma~\ref{odd-permutation}. 
Starting from the obvious equality $\det(g,\cO)=1$, we conclude by induction that $\det(g,\Omega^m(\cO))=1$.
\end{enumerate}
\end{proof}

If $P=C_{2^n}$, it should be noted that there are two natural lifts for~$\Omega^1(k)=\Omega^1_{C_{2^n}/1}(k)$.
One of them is $\Omega^1(\cO)$, but it does not have determinant~1, by Corollary~\ref{omega}.
The other one is $\Phi_{\Omega^1(k)}$, which turns out to be isomorphic to $\cO^- \otimes_\cO \Omega^1(\cO)$,
where $\cO^-$ denotes the one-dimensional module with the generator of~$C_{2^n}$ acting by~$-1$.
Corollary~\ref{omega} together with an induction on~$n$ show that the same holds for the other generators, namely $\Phi_{\Omega^1_{C_{2^n}/Q}(k)}\cong \cO^- \otimes_\cO \Omega^1_{C_{2^n}/Q}(\cO)$ for every $Q\leq C_{2^n}$ of index at least~4.

\vspace{4mm}
\section{Lifting from characteristic~2}

\noindent Our aim is to construct a section for $\pi_p$ using a set of generators for the Dade group $D_k(P)$.
Since the assignment $M\mapsto \Phi_M$ has a good multiplicative behaviour by Lemma~\ref{Phi}(b), one might expect that the multiplicative order in the Dade group is preserved, but this is in fact not at all straightforward. If $[M]$ has order~$n$ in $D_k(P)$, then $M^{\otimes n}$ is isomorphic to a permutation module~$kX$, where $X$ is a basis of~$kX$ permuted by the action of~$P$. Then we obtain
$$(\Phi_M)^{\otimes n} \cong \Phi_{M^{\otimes n}} \cong \Phi_{kX} \,,$$
but $\Phi_{kX}$ may not be a permutation $\cO P$-lattice. The only obvious thing is that $[\Phi_{kX}]$ lies in $\ker(\pi_p) \cong X(P)$.
However, we now show that a much better result holds when the order~$n$ is a power of~2.

\begin{lem} \label{order2}
Let $M$ be a strongly capped endo-permutation $kP$-module. 
\begin{enumerate}[ \rm(a)]
\item If $[M]$ has order~2 in~$D_k(P)$, then $[\Phi_M]$ has order~2 in~$D_\cO(P)$.
\item If $[M]$ has order~4 in~$D_k(P)$, then $[\Phi_M]$ has order~4 in~$D_\cO(P)$.
\end{enumerate}
\end{lem}

\begin{proof}
\begin{enumerate}[ \rm(a)]
\item Since $[M]$ has order~2 in~$D_k(P)$, we have $M\cong M^*$. By Lemma~\ref{Phi}(a), we obtain $\Phi_M\cong \Phi_{M^*}\cong \Phi_M^*$. 
Hence $[\Phi_M]=[\Phi_M]^{-1}$ in~$D_\cO(P)$ and the result follows.
\item Let $N=M\otimes_k M$. Since $[M]$ has order~4 in~$D_k(P)$, $[N]$ has order~2 and so $[\Phi_N]$ has order~2 in~$D_\cO(P)$ by~(a).
Therefore, using Lemma~\ref{Phi}, we obtain
$$[\Phi_M]^4 = [(\Phi_M)^{\otimes 4}] = [\Phi_{M^{\otimes 4}}] = [\Phi_{N^{\otimes 2}}] = [\Phi_N\otimes_\cO \Phi_N] = [\Phi_N]^2 = 1 \,,$$
as required.
\end{enumerate}
\end{proof}

We can now prove Theorem~\ref{main} in the general case.

\begin{proof}[Proof of Theorem~\ref{main}]{\ }
\begin{enumerate}[ \rm(a)]
\item First of all, the case $p\geq 3$ is proved in Lemma~\ref{odd}. Hence, we may assume that $p=2$.
We rely on the fact that
$$D_k(P)\cong (\IZ/2\IZ)^a\times (\IZ/4\IZ)^b\times \IZ^c$$
for some non-negative integers $a,b,c$ (see \cite[Theorem~14.1]{ThevenazTour}).
This follows either from Bouc's classification of endo-permutation modules \cite[Section 8]{Bouc}, or from the detection theorems of \cite{CaTh05}
which do not depend on the full classification. In any case, the result uses a reduction to the cases of cyclic 2-groups, semi-dihedral 2-groups and generalised quaternion 2-groups, and in these cases the torsion subgroup of the Dade group over~$k$ only contains elements of order 2 or~4 (see \cite{Ca-Th}).\\
Now we choose a generator for each factor $\IZ/2\IZ$, $\IZ/4\IZ$, or $\IZ$. The class of each of them can be lifted to an element of~$D_\cO(P)$ of the same order, by Lemma~\ref{order2}.
This procedure for the generators then extends obviously to a group homomorphism $D_k(P) \to D_\cO(P)$ which is a group-theoretic section of the surjection~$\pi_p$.
\item  By Lemma~\ref{lem:one-dim}, $\Ker(\pi_p)$ is isomorphic to the group $X(P)$ of all one-dimensional $\cO P$-lattices.
Thus the isomorphism $D_\cO(P)\cong X(P)\times D_k(P)$ follows from~(a).
\end{enumerate}
\end{proof}

\begin{rem}[Non-uniqueness of the section]
It is clear that the group-theoretic section of $\pi_p$ obtained through Lemma~\ref{odd} in odd characteristic and in the proof of Theorem~\ref{main} in characteristic~2 is not unique in general, because it suffices to send generators of $D_k(P)$ to lifts of the same order in $D_{\cO}(P)$.  Thus the number of group-theoretic sections of $\pi_p$  depends on the structure of the kernel $\Ker(\pi_p)\cong X(P)$.
However,  in case $p$ is odd, we note that the section is uniquely determined on the torsion part $D^t_k(P)$ of~$D_k(P)$, because  $D^t_k(P)$ is a $2$-group and $X(P)$ a $p$-group.
\end{rem}


\bigskip\bigskip\bigskip
\noindent\textbf{Acknowledgments.}
This paper is based upon work supported by the National Science Foundation under Grant No.~DMS1440140
while the first author was in residence at the Mathematical Sciences Research Institute in Berkeley,
California, during the Spring 2018 semester. The first author also acknowledges funding  by the DFG SFB TRR 195.
\par
The authors thank the referee for pointing out an error in an earlier version of this paper.

\bigskip\bigskip\bigskip



\end{document}